\documentclass[11pt,intlimits]{article}
\usepackage[centertags]{amsmath}
\usepackage[english,german,french]{babel}
\usepackage{amsfonts}
\usepackage{dsfont}
\usepackage{amssymb,amssymb,amscd}
\usepackage{amsthm}
\usepackage[T1]{fontenc}
\usepackage{mathbbol} 
\usepackage{a4wide}
\usepackage{mathrsfs}

\usepackage{color}
\usepackage{stackrel}
\usepackage{authblk}

\usepackage{multicol}

\usepackage{tikz-cd}
\usepackage{tikz,lipsum}
\usetikzlibrary{shapes}
\usetikzlibrary{decorations.pathmorphing}
\usetikzlibrary{arrows}
\usetikzlibrary{decorations.markings}
\usetikzlibrary{backgrounds,calc}
\tikzset{middlearrow/.style={
        decoration={markings,
            mark= at position 0.6 with {\arrow{#1}} ,
        },
        postaction={decorate}
    }
}
\usepackage{capt-of}
\usepackage{environ}
\makeatletter
\newsavebox{\measure@tikzpicture}
\NewEnviron{scaletikzpicturetowidth}[1]{t%
  \def\tikz@width{#1}%
  \begin{lrbox}{\measure@tikzpicture}%
  \BODY
  \end{lrbox}%
  \pgfmathparse{#1/\wd\measure@tikzpicture}%
  \BODY
}
\makeatother

\usepackage{dynkin-diagrams}
\makeatletter
\newcommand{\extraNode}[6]%
{%
\dynkinPlaceRootRelativeTo{#1}{#2}{#3}{#4}{#5}
\dynkinDefiniteSingleEdge{#1}{#2}
\dynkinRootMark{o}{#1}
\advance\dynkin@nodes by 1
\dynkinLabelRoot{#1}{#6} 
}%
\newcommand{\extraDotNode}[6]%
{%
\dynkinPlaceRootRelativeTo{#1}{#2}{#3}{#4}{#5}
\dynkinIndefiniteSingleEdge{#1}{#2}
\dynkinRootMark{o}{#1}
\advance\dynkin@nodes by 1
\dynkinLabelRoot{#1}{#6} 
}%
\makeatother
\tikzset{/Dynkin diagram,mark=o,edge length=.5cm}

\graphicspath{{Fig/}}

\newtheorem{theorem}{Theorem}
\newtheorem{proposition}[theorem]{Proposition}%
\newtheorem{lemma}[theorem]{Lemma}
\theoremstyle{thmstyletwo}%
\newtheorem{example}{Example}%
\newtheorem{remark}{Remark}%
\theoremstyle{thmstylethree}%
\newtheorem{definition}{Definition}

\newtheorem{conjecture}{Conjecture}

\usepackage{comment}
\usepackage{tikz-cd}
\usepackage{tikz}
\usetikzlibrary{graphs,quotes}
\usepackage[all]{xy}  

\date{}

\begin{document}
\selectlanguage{english}


\title{Components of $V(m\rho) \otimes V(n\rho)$}

\author[1,2]{Rekha Biswal}
\author[3]{Sam Jeralds}

\affil[1]{School of Mathematical Sciences, NISER, Bhubaneswar 752050, India}
\affil[2]{Homi Bhabha National Institute, Training School Complex, Anushakti Nagar, Mumbai 400094, India}
\affil[3]{University of Sydney, Camperdown NSW 2006, Australia}






\maketitle

\begin{abstract}
Let $\mathfrak{g}$ be a symmetrizable Kac-Moody Lie algebra and let $\rho$ denote the sum of the fundamental weights. The irreducible highest weight representations $V(m\rho)$ occupy a distinguished position in representation theory due to their rich symmetry and geometric significance. In this paper, we study the tensor products
\[
V(m\rho)\otimes V(n\rho), \quad m,n \in \mathbb{N},
\]
and investigate the structure of their irreducible decompositions. Motivated by the classical conjecture of Kostant, which predicts a highly structured behavior in simpler settings, we propose a general framework describing the irreducible components appearing in such tensor products for finite-dimensional semisimple or affine Kac-Moody Lie algebras $\mathfrak{g}$. Our results identify a family of dominant weights governing the decomposition and provide criteria for their occurrence. This work extends the scope of Kostant-type phenomena and reveals new structural patterns in tensor products associated with multiples of the Weyl vector.
\end{abstract}

\section{Introduction}
\subsection{Kostant's $V(\rho) \otimes V(\rho)$ conjecture}
The representation theory of complex semisimple Lie algebras is a central area of modern mathematics, with deep connections to algebraic geometry, combinatorics, and mathematical physics. A fundamental problem in this theory is to understand the decomposition of tensor products of irreducible representations into irreducible components. While general tools such as the Weyl character formula and crystal bases provide theoretical answers, explicit and uniform descriptions of tensor product decompositions remain difficult to obtain in most cases.

Let $\mathfrak{g}$ be a complex semisimple Lie algebra with root system $\Phi$, and let $\rho$ denote the half-sum of positive roots. For each dominant integral weight $\lambda \in P^+$, we denote by $V(\lambda)$ the irreducible highest weight representation of highest weight $\lambda$. Among these, the representations $V(m\rho)$, for $m \in \mathbb{N}$, are of particular importance. They arise naturally in harmonic analysis, the study of primitive ideals, and the geometry of flag varieties, and reflect deep symmetries inherent in the structure of $\mathfrak{g}$.

A classical conjecture of Kostant \cite{BZ} suggests that tensor products involving representations related to $\rho$ exhibit remarkable regularity: all irreducible summands which can appear must actually appear.

\begin{conjecture} \label{KosConj}
    Let $\mathfrak{g}$ be a complex semisimple Lie algebra, and $\lambda \in P^+$ a dominant integral weight such that $\lambda \leq 2\rho$ in the dominance order. Then $V(\lambda) \subseteq V(\rho) \otimes V(\rho)$. 
\end{conjecture}

This conjecture was first proven for $\mathfrak{g}=sl(n+1)$ by Berenstein-Zelevinsky \cite{BZ}, where the conjecture first appeared in the literature. Chiriv\'i-Kumar-Maffei \cite{CKM} showed that this conjecture holds for all $\mathfrak{g}$, up to a saturation factor (cf. Definition \ref{satfact}), recovering the theorem for $sl(n+1)$. The conjecture was later extended by Kumar and the second author to affine Kac-Moody Lie algebras (where $\rho$ is interpreted as the sum of the fundamental weights) and shown to hold for $\mathfrak{g}=\widehat{sl}(n+1)$, and again for all affine Kac-Moody Lie algebras up to a saturation factor.

To date, Conjecture \ref{KosConj} remains open for all other classical simple Lie algebras, having been checked directly for $\mathfrak{g}$ of exceptional types. 


\subsection{A generalization to unequal parameters}

By the same proofs as in \cite{CKM, JK}, Conjecture \ref{KosConj} can easily be extended to tensor products of the form $V(n\rho) \otimes V(n\rho)$ for any integer $n \geq 1$, with the condition that $\lambda \leq 2n\rho$ being sufficient. We refer to this as the ``equal parameter" or ``diagonal" case. 

In this note, we propose a generalization of Conjecture \ref{KosConj} to include the \textit{unequal parameter} or non-diagonal case by considering tensor products of the form 
\[
V(m\rho)\otimes V(n\rho)
\]
for $m \geq n \geq 0$. In order to pose a generalization of Kostant's conjecture that covers the non-diagonal case, we first need a suitable replacement for the proposed sufficient condition $\lambda \leq 2\rho$. The first natural guess, requiring $\lambda \leq (m+n)\rho$, turns out to be insufficient, as in the following example. 

\begin{example}
   Let $\mathfrak{g}$ be of type $B_2$. Then there are sixteen dominant integral weights $\lambda \leq 7\rho$ such that $V(\lambda) \not \subseteq V(5\rho) \otimes V(2\rho)$. One such example is $\lambda=\omega_2$. 
\end{example}

Instead, we are motivated by a different necessary condition on tensor product decompositions; namely, if $V(\nu) \subseteq V(\lambda) \otimes V(\mu)$ for dominant integral weights $\lambda, \mu, \nu$, then necessarily 
$$
\nu=\lambda+\beta
$$
where $\beta$ is some weight appearing in the representation $V(\mu)$. By (\cite{CKM}, Proposition 9), a dominant integral weight $\lambda$ satisfies $\lambda \leq 2\rho$ if and only if $\lambda=\rho+\beta$ where $\beta$ is a weight of $V(\rho)$. Thus, we make the following modified conjecture, which is the topic of this note. 

\begin{conjecture} \label{RekhaConj}
Let $\mathfrak{g}$ be a simple Lie algebra or untwisted affine Kac-Moody Lie algebra. Let $m \geq n \geq 0$ be two nonnegative integers, and let $\lambda \in P^+$ be a dominant integral weight such that 
$$
\lambda=m\rho+\beta
$$
for some weight $\beta$ of the representation $V(n\rho)$. Then $V(\lambda) \subseteq V(m\rho) \otimes V(n\rho)$. 
\end{conjecture}

That is, the dominant weights occurring in the tensor product appear to be obtained simply by translating the weight set of $V(n\rho)$ by $m\rho$ and considering the resulting dominant weights. Such a description is unexpected to hold for tensor products at large. 

\begin{remark}
   Note that the statement of Conjecture \ref{RekhaConj} is equally valid for $\mathfrak{g}$ of any symmetrizable Kac-Moody type; we restrict ourselves only to the finite-dimensional simple Lie algebras or untwisted affine Kac-Moody Lie algebras for familiarity, and as our methods at present rely on these restrictions.  
\end{remark}

As initial evidence to this conjecture, we adapt methods from \cite{CKM} using the saturated tensor cone of $\mathfrak{g}$ to prove the following theorem, which is the main result of this note. 

\begin{theorem} \label{MainTheorem}
\begin{enumerate}
    \item Let $\mathfrak{g}$, $m \geq n \geq 0$, and $\lambda \in P^+$ be as in Conjecture \ref{RekhaConj}. Let $d \geq 1$ be a saturation factor for $\mathfrak{g}$. Then 
$$
V(d\lambda) \subseteq V(dm\rho) \otimes V(dn\rho).
$$
In particular, if $\mathfrak{g}=sl_{n+1}(\mathbb{C})$, we have that $V(\lambda) \subseteq V(m\rho) \otimes V(n\rho)$.

\item If $\mathfrak{g}=A^{(1)}_1=\widehat{sl}_2(\mathbb{C})$, then $V(\lambda) \subseteq V(m\rho) \otimes V(n\rho)$. 
\end{enumerate}
\end{theorem}

In the setting of $\mathfrak{g}$ an affine Kac-Moody Lie algebra, the representation theory exhibits further structure absent in finite type. In particular, the interaction between affine tensor products and a Virasoro algebra symmetry arising from the Goddard-Kent-Olive (GKO) construction naturally leads to questions involving $\delta$-maximal weights and $\delta$-maximal components. Beyond $\mathfrak{g}=A_1^{(1)}$, we partially adapt techniques of \cite{JK} to the current setting to exhibit support for Conjecture \ref{RekhaConj} for $\mathfrak{g}=A_n^{(1)}$ (with no saturation factor). This encounters new difficulties arising from the $m \neq n$ assumption that is not present in the analogue of Kostant's conjecture considered in \cite{JK}.

\subsection{A connection to embeddings of tensor products}
We highlight a compatibility of Conjecture \ref{RekhaConj} and embeddings of tensor products of irreducible representations. The latter question originates in the study of fusion products of cyclic $\mathfrak{g}[t]$-modules, where $\mathfrak{g}[t]:=\mathfrak{g} \otimes_\mathbb{C} \mathbb{C}[t]$ is the current algebra associated to a simple Lie algebra. See \cite{BS, CFS, DP, LPP} for an overview of the following problem.

Let $\lambda, \mu, \lambda', \mu' \in P^+$ be four dominant integral weights such that $\lambda+\mu=\lambda'+\mu'$. Then we say that $(\lambda, \mu) \preceq (\lambda', \mu')$ if 
$$
\mathrm{min}\{\lambda(\beta^\vee), \mu(\beta^\vee)\} \leq \mathrm{min}\{\lambda'(\beta^\vee), \mu'(\beta^\vee)\}
$$
for all positive coroots $\beta^\vee$. The following conjecture relates this ordering on pairs of dominant weights to embeddings of tensor products.

\begin{conjecture} \label{TensorEmbed}
    Suppose $(\lambda, \mu) \preceq (\lambda', \mu')$. Then $V(\lambda) \otimes V(\mu) \subseteq V(\lambda') \otimes V(\mu')$. 
\end{conjecture}

The above ``Schur positivity" conjecture was originally formulated in \cite{LPP} for the Lie algebra $\mathfrak{sl}_n$. It was subsequently extended to arbitrary simple Lie algebras in \cite{CFS}, where it was verified in the case of $\mathfrak{sl}_3$ and for multiples of minuscule fundamental weights using crystal-theoretic methods. A weaker form of the conjecture, namely support containment, was established in \cite{DP}. Further progress in this direction was made in \cite{BS}, where support containment was proved for certain special cases of simply laced symmetrizable Kac–Moody algebras, again via crystal theory. Finally, in \cite{Speyer}, a complete proof of the conjecture for $\mathfrak{sl}_n$ was obtained through the construction of a new combinatorial model for Littlewood–Richardson coefficients. 

Now, suppose that Conjecture \ref{RekhaConj} holds (in either finite or affine types). For $m>n$, note that $(m\rho, n\rho) \preceq ((m-1)\rho, (n+1)\rho)$, thus Conjecture \ref{TensorEmbed} suggests that $V(m\rho) \otimes V(n\rho) \subseteq V((m-1)\rho) \otimes V((n+1)\rho)$. On the other hand, if $\lambda=m\rho+\beta$ for some weight $\beta$ of $V(n\rho)$ as in the assumptions of Conjecture \ref{RekhaConj}, then as $m>n$ we can write 
$$
\lambda=(m-1)\rho+(\rho+\beta)=:(m-1)\rho+\beta',
$$
where now $\beta'=\rho+\beta$ is a weight of $V((n+1)\rho)$. In particular, we have the weaker ``support containment" result for these pairs of weights, as in the following 

\begin{proposition}
    Assume Conjecture \ref{RekhaConj} holds, and that $m>n$. Then $V(\lambda) \subseteq V(m\rho) \otimes V(n\rho) \implies V(\lambda) \subseteq V((m-1)\rho) \otimes V((n+1)\rho).$
\end{proposition}
It would be interesting, given the distinguished role that multiples of $\rho$ play in representation theory and tensor product decomposition problems, to establish not just support containment but embeddings for these pairs assuming Conjecture \ref{RekhaConj}.

\subsection{Outline of the paper}

In Section \ref{TensorCone}, we review the saturated tensor cone for simple Lie algebras and affine Kac-Moody Lie algebras, and complete the proof of Theorem \ref{MainTheorem}(1). In Section \ref{Vir}, we recall the interactions between the saturated tensor cone for affine Kac-Moody Lie algebras and the Virasoro algebra via the Goddard-Kent-Olive construction, and give a proof of Theorem \ref{MainTheorem}(2), along with an outlook for how to adapt these methods for $\mathfrak{g}=\widehat{sl}_{n+1}(\mathbb{C})$ in general. In an appendix, we collect some examples in support of Conjecture \ref{RekhaConj} in small rank and its relationship to tensor product embeddings.

\section{Tensor cones and inequalities} \label{TensorCone}

We fix $\mathfrak{g}$ to be either a simple (finite-dimensional) Lie algebra or an untwisted affine Kac-Moody Lie algebra associated to a finite-dimensional simple Lie algebra. While the proofs of Theorem \ref{MainTheorem}(1) in each case are identical, we will adopt the notational conventions of \cite{BK} and \cite{Res} from the affine Kac-Moody setting. 

\begin{definition}
The saturated tensor cone $\Gamma_2(\mathfrak{g})$ is given by the set 
$$
\Gamma_2(\mathfrak{g}):=\{(\lambda, \mu, \nu)\in (P^+)^3 : \exists N \geq 1 \text{ with } V(N \nu) \subseteq V(N \lambda) \otimes V(N\mu) \}.
$$
\end{definition}
The saturated tensor cone often appears under the name the saturated tensor semigroup or the additive eigencone of $\mathfrak{g}$, and has a natural generalization to include more than two tensor factors. For $\mathfrak{g}=sl_{n+1}(\mathbb{C})$, this recovers the classical Horn cone related to the Hermitian additivie eigenvalue problem. 

When $\mathfrak{g}$ is finite-dimensional, this is finitely-generated polyhedral cone; in the affine setting, the cone is no longer finitely generated and is locally polyhedral. In the latter case, the existence of a saturation factor, as in the following definition, is not immediate but has been determined by Ressayre \cite{Res}.

\begin{definition} \label{satfact}
    A saturation factor $d \geq 1$ of $\Gamma_2(\mathfrak{g})$ is a positive integer such that, for any $(\lambda, \mu, \nu) \in (P^+)^3$ where $\lambda+\mu-\nu \in Q$ and $(\lambda, \mu, \nu) \in \Gamma_2(\mathfrak{g})$, we have $V(d \nu) \subseteq V(d\lambda) \otimes V(d\mu)$.
\end{definition}

For $\mathfrak{g}=sl_{n+1}(\mathbb{C})$, it is known by the saturation theorem of Knutson-Tao \cite{KT} that in that case we can take $d=1$ as a saturation factor. In general, $d >1$, and the determination of minimal saturation factors (and when one can take saturation factor $d=1$) is an open area of research.

In the cases considered here, the saturated tensor cone $\Gamma_2(\mathfrak{g})$ has an explicit description in terms of inequalities. For finite-dimensional $\mathfrak{g}$, the following presentation is due to Belkale-Kumar \cite{BeK}. In the affine Kac-Moody setting, these inequalities were conjectured for all symmetrizable Kac-Moody types by Brown-Kumar \cite{BK} and proven for $\mathfrak{g}=\widehat{sl}_{2}(\mathbb{C})$. The proof for all affine Kac-Moody $\mathfrak{g}$ is due to Ressayre \cite{Res}. 

\begin{theorem} \label{Ineq}
Let $x_i$ be the fundamental coweights of $\mathfrak{g}$ dual to the simple roots, defined by $\alpha_j(x_i)=\delta_{i,j}$. Fix $(\lambda, \mu, \nu) \in (P^+)^3$ with $\lambda+\mu-\nu \in Q$. If $\mathfrak{g}$ is an affine Kac-Moody Lie algebra, assume further that $\lambda(c)>0$ and $\mu(c)>0$, where $c$ is the central element of $\mathfrak{g}$. Then the following are equivalent: 
\begin{enumerate}
    \item $(\lambda, \mu, \nu) \in \Gamma_2(\mathfrak{g})$.
    \item For all maximal parabolic subgroups $P_i$ of $G^{min}$, the minimal Kac-Moody group associated to $G$, and triples $(u,v,w) \in (W^{P_i})^3$ of minimal-length coset representatives such that the cohomology class $\varepsilon^w \in H^\ast(G/P_i, \mathbb{Z})$ appears with coefficient $1$ in the deformed product $\varepsilon^u \odot_0 \varepsilon^v$, the following inequality is satisfied:
    $$
    \nu(w x_i) \leq \lambda(u x_i) + \mu(v x_i).
    $$
\end{enumerate}
\end{theorem}

By Theorem \ref{Ineq}, to prove Theorem \ref{MainTheorem}(1) it suffices to show that for any $m \geq n \geq 0$ and $\lambda \in P^+$ of the prescribed form that $(m\rho, n\rho, \lambda) \in \Gamma_2(\mathfrak{g})$ or equivalently that $(m\rho, n\rho, \lambda)$ satisfies all of the inequalities determining $\Gamma_2(\mathfrak{g})$. Note that when $\mathfrak{g}$ is of affine Kac-Moody type and we are in the nontrivial case of $m \geq n >0$, the condition on the central charge $m\rho(c)>0$ and $n\rho(c)>0$ is satisfied. 

By definition, for any dominant integral weight $\lambda \in P^+$ we have that $V(0) \otimes V(\lambda) \cong V(\lambda)$, so $(0, \lambda, \lambda) \in \Gamma_2(\mathfrak{g})$. This is also directly observable from the inequalities, as in the following lemma. 

\begin{lemma} \label{ConeTranslate}
For a maximal parabolic subgroup $P_i$, let $(u,v,w)\in (W^{P_{i}})^3$ be such that $\varepsilon^w$ appears with coefficient $1$ in the deformed product $\varepsilon^u \odot_0 \varepsilon^v$. Then both  $\lambda(v x_i) - \lambda(w x_i) \geq 0$ and $\lambda(u x_i)-\lambda(w x_i) \geq 0$. 
\end{lemma}
\begin{proof}
    Since $\varepsilon^w$ appears in
$
\varepsilon^u \odot_0 \varepsilon^v,
$
it also appears in the ordinary cup product
\[
\varepsilon^u \cdot \varepsilon^v
\]
in $H^*(G/P_i)$. Recall that the Schubert basis satisfies
\[
\varepsilon^u \cdot \varepsilon^v
 = \sum_{y\in W^{P_i}} c_{u,v}^y \,\varepsilon^y,
\]
where $c_{u,v}^y\neq 0$ only if
\[
y \succeq u,
\qquad
y \succeq v
\]
in the Bruhat order on $W^{P_i}$ (cf. \cite{Kumar} Corollary 11.3.17). In particular, since the coefficient of
$\varepsilon^w$ is nonzero, we have
\[
w \succeq u, v.
\]

Now recall the standard fact that if $
y \succeq z
$
in Bruhat order, then for every dominant coweight $\eta$,
\[
z\eta-y\eta \in Q_+^\vee.
\]
Applying this with $y=w$, $z=u$, and $\eta=x_i$, we obtain
\[
ux_i-wx_i \in Q_+^\vee.
\]
Since $\lambda$ is dominant integral, it is nonnegative on positive
coroots. Therefore,
\[
\lambda(ux_i-wx_i)\ge 0,
\]
and hence
\[
\lambda(ux_i)-\lambda(wx_i)\ge 0.
\]
The same follows for $v$ by symmetry in $u$ and $v$.
\end{proof}

We will make crucial use of the following proposition, which is taken from (\cite{BK}, Proposition 7.1(b)).

\begin{proposition} \label{BaseIneq}
    Let $(u,v,w)\in (W^{P_i})^3$ such that $\varepsilon^w$ appears with coefficient $1$ in $\varepsilon^u \odot_0 \varepsilon^v$. Then 
    $$
    (u^{-1}\rho +v^{-1}\rho -w^{-1}\rho-\rho)(x_i) \geq 0.
    $$
    (In fact, in this case the stronger statement that this is precisely equal to $0$ holds, but we will not need this stronger result.)
\end{proposition}

With this in hand, we are now ready to prove Theorem \ref{MainTheorem}(1).

\begin{proof}[Proof of Theorem \ref{MainTheorem}(1)]
Let $m \geq n \geq 0$ and let $\lambda \in P^+$ be a dominant integral weight such that $\lambda=m\rho+\beta$ for some weight $\beta$ of the representation $V(n\rho)$. Fix a maximal parabolic subgroup $P_i$, and a triple $(u,v,w) \in (W^{P_i})^3$ such that $\varepsilon^w$ appears with coefficient $1$ in the deformed product $\varepsilon^u \odot_0 \varepsilon^v$. We compute directly that 
$$
\begin{aligned}
m\rho(u x_i)+&n\rho(vx_i)-\lambda(w x_i) = m\rho(ux_i)+n\rho(vx_i)-m\rho(w x_i) - \beta(w x_i) \\
&=n(u^{-1} \rho(x_i)) +n(v^{-1}\rho(x_i))-n(w^{-1}\rho(x_i)) -w^{-1}\beta(x_i)  + (m-n) \left(\rho(u x_i) - \rho(wx_i) \right)
\end{aligned}
$$
Now, as $\beta$ is a weight of $V(n\rho)$, so is $w^{-1}\beta$, so that $$n\rho-w^{-1}\beta \in Q^+.$$ In particular, $(n\rho-w^{-1}\beta)(x_i) \geq 0$, so that $w^{-1}\beta(x_i) \leq n\rho(x_i)$. Thus
$$
m\rho(ux_i)+n\rho(vx_i)-\lambda(wx_i) \geq n \left(u^{-1}\rho+v^{-1}\rho-w^{-1}\rho-\rho\right)(x_i) + (m-n)(\rho(ux_i)-\rho(wx_i)).
$$
By Lemma \ref{ConeTranslate}, Proposition \ref{BaseIneq}, and that $m \geq n \geq 0$, both terms on the right hand side of this inequality are nonnegative, so that 
$$
m\rho(ux_i)+n\rho(vx_i)-\lambda(wx_i) \geq 0,
$$
hence $(m\rho, n\rho, \lambda) \in \Gamma_2(\mathfrak{g})$ as desired. 
\end{proof}

\begin{remark}
    For any symmetrizable Kac-Moody Lie algebra $\mathfrak{g}$, an analogous set of inequalities are necessary for points $(\lambda, \mu, \nu) \in \Gamma_2(\mathfrak{g})$ (\cite{BK} Theorem 1.1). That these inequalities are also sufficient is still an open question. Assuming this, the above proof would apply mutatis mutandis to get the analogue of Theorem \ref{MainTheorem}(1) in the more general symmetrizable setting.
\end{remark}

While the above result only demonstrates Conjecture \ref{RekhaConj} up to saturation, many direct examples can be computed to support the conjecture without saturation factor. In the $m=n=1$ case of the original Kostant conjecture, this has been done for all exceptional types, as reported by \cite{CKM}. We give some select examples of $m >n >0$ to demonstrate the feasibility of the conjecture in the Appendix.

\section{Virasoro action and affine Lie algebras} \label{Vir}

\subsection{Consequences of a Virasoro action on tensor products}

In this section, we focus on the case when $\mathfrak{g}$ is an untwisted affine Kac-Moody Lie algebra. In this setting, the interaction between the action of $\mathfrak{g}$ and the action of the Virasoro algebra on $\mathfrak{g}$-representations provides fruitful information on branching problems, particularly the decomposition of tensor products. This approach has been utilized many times in the literature; cf. \cite{KW, KRR, BK, Res, JK2, JK} for many such examples and references. 

We collect here the necessary technical results needed for applying the representation theory of the Virasoro algebra to tensor products of $\mathfrak{g}$-representations, following (\cite{KRR} Lecture 10). The action arises from the Goddard-Kent-Olive (or GKO) construction; this is a coset variation of the usual Sugawara construction of the Virasoro algebra in the completion of the enveloping algebra $U(\mathfrak{g})$. We denote by $Vir_{GKO}=\langle L_k^{GKO} (k \in \mathbb{Z}), C \rangle$ the Virasoro algebra arising from this construction.

\begin{proposition} \label{GKO}
    Let $\mathfrak{g}$ be an untwisted affine Kac-Moody Lie algebra and $\mathring{\mathfrak{g}}$ the underlying finite-dimensional simple Lie algebra. Let $\lambda, \mu \in P^+$ be two dominant integral weights of $\mathfrak{g}$ with positive levels $\lambda(c)=\ell, \mu(c)=m$. Then
    \begin{enumerate}
        \item $V(\lambda) \otimes V(\mu)$ is a unitarizable Virasoro representation with nonnegative central charge 
        $$
        (\dim \mathring{\mathfrak{g}}) \left(\frac{\ell}{\ell+h^\vee} + \frac{m}{m+h^\vee}-\frac{\ell+m}{\ell+m+h^\vee} \right)$$
        where $h^\vee$ is the dual Coxeter number of $\mathfrak{g}$.

        \item $L_0^{GKO}$ acts on $V(\lambda) \otimes V(\mu)$ by 
        $$
        \frac{1}{2} \left( \frac{(\lambda | \lambda+2\rho)}{\ell+h^\vee}+ \frac{(\mu | \mu+2\rho)}{m+h^\vee} - \frac{\Omega}{\ell+m+h^\vee} \right),
        $$
        where $(\cdot | \cdot)$ is the normalized invariant form on $\mathfrak{h}^\ast$ and $\Omega$ is the Casimir operator.

        \item For all $k$, $[L_k^{GKO}, \mathfrak{g}']=0$; that is, $L_k^{GKO}$ are intertwining operators for the $\mathfrak{g}'=[\mathfrak{g},\mathfrak{g}]$ action on $V(\lambda) \otimes V(\mu)$. 

    \end{enumerate}
\end{proposition}

We now recall two notions of ``$\delta$-strings" in the representation theory of $\mathfrak{g}$ and their maximal elements. First, if $\nu$ is a weight appearing in the representation $V(\lambda)$, then there is some $N \in \mathbb{Z}$ such that the set 
$$
\{\nu+k \delta: k \leq N\}
$$
all appear as weights in $V(\lambda)$. We call the weight $\nu+N\delta$ the \textbf{$\delta$-maximal weight} along the string through $\nu$. Similarly, we say that $V(\nu+N\delta) \subseteq V(\lambda) \otimes V(\mu)$ is a \textbf{$\delta$-maximal component} if $V(\nu+k\delta) \not \subset V(\lambda) \otimes V(\mu)$ for any $k>N$. For a fixed $\delta$-maximal component (appropriately relabeling the highest weight) $V(\nu) \subseteq V(\lambda) \otimes V(\mu)$, the implications of Proposition \ref{GKO} give the following structure theory for $\delta$-strings in $\Gamma_2(\mathfrak{g})$. 

\begin{proposition} \label{deltaMax}
Retaining notation as in Proposition \ref{GKO}, let $V(\nu)$ be a $\delta$-maximal component of $V(\lambda) \otimes V(\mu)$. Then 
\begin{enumerate}
    \item If $\frac{1}{2} \left( \frac{(\lambda | \lambda+2\rho)}{\ell+h^\vee}+ \frac{(\mu | \mu+2\rho)}{m+h^\vee} - \frac{(\nu | \nu+2\rho)}{\ell+m+h^\vee} \right)>0$, then $V(\nu-k\delta) \subseteq V(\lambda) \otimes V(\mu)$ for all $k \geq 0$. 
    \item Otherwise, if $\frac{1}{2} \left( \frac{(\lambda | \lambda+2\rho)}{\ell+h^\vee}+ \frac{(\mu | \mu+2\rho)}{m+h^\vee} - \frac{(\nu | \nu+2\rho)}{\ell+m+h^\vee} \right)=0$, then $V(\nu-k\delta) \subseteq V(\lambda) \otimes V(\mu)$ for $k=0$ and $k \geq 2$.
\end{enumerate}
\end{proposition}

That is, the action of $L_0^{GKO}$ on a $\delta$-maximal component detects if the string is unbroken; these two cases are the only phenomena that can occur along these strings (cf. \cite{JK2} Section 3). Note, however, that this does not determine whether a given $\delta$-maximal component appears in the tensor product nor determines the $\delta$-maximal components themselves. Nevertheless, a result of Kac-Wakimoto \cite{KW} shows that \textit{some} $\delta$-maximal component must appear: 

\begin{lemma} \label{KWLem}
Let $\lambda, \mu, \nu \in P^+$ such that $\lambda+\mu-\nu \in Q$. Then there is some $k \geq 0$ such that $V(\nu-k\delta) \subseteq V(\lambda) \otimes V(\mu)$. 
\end{lemma}

\subsection{The Virasoro action on $V(m \rho) \otimes V(n\rho)$}

We specialize to the case of $V(m\rho) \otimes V(n\rho)$. Consider any dominant integral weight of the form $\lambda=m\rho + \beta$ where $\beta$ is a weight of $V(n\rho)$. To apply Proposition \ref{deltaMax}, we carefully compute and arrive at the following lemma. 

\begin{lemma} \label{GKOIneq}
For $\lambda$ as above, $$\frac{1}{2h^\vee} \left( \frac{(m\rho | (m+2)\rho)}{m+1} + \frac{(n\rho | (n+2)\rho)}{n+1} - \frac{(\lambda | \lambda+2\rho)}{m+n+1} \right) >0.$$ 
Consequently, for any component $V(\lambda) \subseteq V(m\rho) \otimes V(n\rho)$ (not necessarily $\delta$-maximal), the operator $L_0^{GKO}$ acts by a positive scalar. 
\end{lemma}

\begin{proof}
    Set $D:=2h^\vee (m+1)(n+1)(m+n+1)$. Then one can expand the above express to get that $L_0^{GKO}$ acts by 
    $$
    \begin{aligned}
    \frac{1}{D} ( (m+1)(n+1)(2\rho | n\rho-\beta) &+ (m+1)(n+1) \{ (n\rho | n\rho)-(\beta|\beta)\} \\
    &+2(m+1)(n+1)(m\rho | n\rho-\beta) + (m+n+2)(m\rho | n\rho)).
    \end{aligned}
    $$
    Now, $(2\rho | n\rho-\beta) \geq 0$ since $2\rho \in P^+$ and $n\rho-\beta \in Q^+$; $(n\rho | n\rho) - (\beta | \beta) \geq 0$ since $\beta$ is a weight of $V(n\rho)$ (cf. \cite{Kac} Proposition 11.4); $(m\rho | n\rho-\beta) \geq 0$ as $n\rho-\beta \in Q^+$ and $m\rho \in P^+$; and finally $(m\rho | n\rho) >0$ by direct computation. Thus in total, this scalar is strictly positive. 
\end{proof}

Note that in the proof of Lemma \ref{GKOIneq}, there was no dependence on where the weight $\beta$ appeared in the $\delta$-string of weights in $V(n\rho)$. As shifting by $\delta$ does not affect the dominance of $\lambda$, an immediate consequence of Lemma \ref{GKOIneq} and Proposition \ref{deltaMax} is the following equivalent statement of Conjecture \ref{RekhaConj} for the affine setting. 

\begin{conjecture} \label{affConjRestate}
Let $m \geq n \geq 0$ and $\lambda=m\rho+\beta$ where $\beta$ is a $\delta$-maximal weight of $V(n\rho)$. Then $V(\lambda) \subseteq V(m\rho) \otimes V(n\rho)$. Equivalently, the $\delta$-maximal components $V(\lambda) \subseteq V(m\rho) \otimes V(n\rho)$ are precisely those $\lambda$ of the form $\lambda=m\rho+\beta$ where $\beta$ is a $\delta$-maximal weight of $V(n\rho)$. 
\end{conjecture}

In general, to prove Conjecture \ref{affConjRestate} one needs to demonstrate the existence of such a component in the tensor product. When $\mathfrak{g}=\widehat{sl}_{2}(\mathbb{C})$, this is handled by Brown-Kumar \cite{BK}, and allows us to prove Theorem \ref{MainTheorem}(2).

\begin{proof}[Proof of Theorem \ref{MainTheorem}(2)]
Let $\mathfrak{g}=\widehat{sl}_2(\mathbb{C})$, and let $\lambda=m\rho+\beta$ be a dominant weight, where $\beta$ is a $\delta$-maximal weight of $V(n\rho)$. By (\cite{BK} Theorem 6.1), the $\delta$-maximal component $V(\lambda+N\delta) \subseteq V(m\rho) \otimes V(n\rho)$ is given by the minimum $N=\mathrm{min}\{k_1, k_2\}$ where $k_1, k_2 \in \mathbb{Z}$ are such that $\lambda-m\rho+k_1 \delta$ is a $\delta$-maximal weight in $V(n\rho)$, and similarly $k_2$ is such that $\lambda-n\rho+k_2\delta$ is $\delta$-maximal in $V(m\rho)$. By assumption on $\beta$, immediately we have that $k_1=0$. Similarly, $k_2 \geq 0$, as $m\rho-(n\rho-\beta)$ is a weight of $V(m\rho)$ and is either already $\delta$-maximal, or otherwise it sits below some $m\rho-(n\rho-\beta)+k_2\delta$ for $k_2 \geq 1$. In particular, $N=0$ and thus $V(\lambda) \subseteq V(m\rho) \otimes V(n\rho)$, as desired.  
\end{proof}
 
To prove Conjecture \ref{affConjRestate} in general, it would be advantageous to have a similar uniform description of all possible $\delta$-maximal components in tensor products. Towards this goal, we prove the following partial result for $\mathfrak{g}=\widehat{sl}_{r+1}(\mathbb{C})$. This produces select $\delta$-maximal components of $V(m\rho) \otimes V(n\rho)$ of the desired form in this setting, as a consequence of the validity of Conjecture \ref{RekhaConj} for $sl_{r+1}(\mathbb{C})$.

\begin{proposition}
    Let $\mathfrak{g}=\widehat{sl}_{r+1}(\mathbb{C})$. Let $\lambda=m\rho+\beta$, where $\beta$ is a \textbf{dominant} $\delta$-maximal weight of the representation $V(n\rho)$. Then $V(\lambda) \subseteq V(m\rho) \otimes V(n\rho)$. 
\end{proposition}

\begin{proof}
    By (\cite{BK} Proposition 4.4), if $\beta \in P^+$ is a $\delta$-maximal weight of $V(n\rho)$, then in the expansion 
    $$
    n\rho-\beta=\sum_{i=0}^r c_i \alpha_i
    $$
    there must be some $c_i$ with $c_i=0$. By symmetry of the Dynkin diagram of $\widehat{sl}_{r+1}(\mathbb{C})$, along with the fact that $m\rho$ and $n\rho$ are also stable under the diagram symmetry, without loss of generality assume $c_0=0$ so that $n\rho-\beta \in \mathring{Q}^+$, the positive part of the finite root lattice. 

    Let $\mathring{V}(m\rho)$ be the $sl_{r+1}(\mathbb{C})$ submodule of $V(m\rho)$ generated by the highest weight vector, and similarly for $\mathring{V}(n\rho)$. Then as $n\rho-\beta \in \mathring{Q}^+$, applying Theorem \ref{MainTheorem}(1) we get that 
    $$
    \mathring{V}(\lambda) \subseteq \mathring{V}(m\rho) \otimes \mathring{V}(n\rho).
    $$
    Let $v_\lambda$ be the $sl_{r+1}(\mathbb{C})$ highest weight vector of $\mathring{V}(\lambda)$. Then in fact this is a $\mathfrak{g}$ highest weight vector; $e_0.v_\lambda=0$ since $(m+n)\rho-\lambda \in \mathring{Q}^+$. Thus this vector generates the subrepresentation 
    $$
    V(\lambda) \subseteq V(m\rho) \otimes V(n\rho)
    $$
    as desired. 
\end{proof}

\section*{Acknowledgments}
R.B. gratefully acknowledges Arkady Berenstein for insightful discussions during her visit to Institut des Hautes Études Scientifiques in summer 2025, which played an important role in the formulation of the main conjecture of this paper. She also thanks Institut des Hautes Études Scientifiques for its warm hospitality and excellent working environment, where part of this work was carried out. Financial support from the National Institute of Science Education and Research, Bhubaneswar, and the Homi Bhabha National Institute, Mumbai, is also gratefully acknowledged.

\section*{Appendix: Examples in Support of Conjecture \ref{RekhaConj}} 
We give here some select examples to support Conjecture \ref{RekhaConj} beyond $sl_{n+1}(\mathbb{C})$. We also highlight in the final example the relationship between Conjecture \ref{RekhaConj} and the problem of embedding tensor products of irreducible representations. 

\begin{example}
Let $\mathfrak g$ be of type $B_2$, and consider the tensor product
\[
V(5\rho)\otimes V(2\rho).
\]
Since $\rho=\omega_1+\omega_2$, this corresponds to
\[
V(5,5)\otimes V(2,2)
\]
in fundamental weight coordinates.

Using SageMath, we compute all dominant weights of the form
\[
\lambda = 5\rho+\beta,
\]
where $\beta$ is a weight of $V(2\rho)$. The resulting dominant weights are
\[
\begin{aligned}
&(3,3), (2,5), (1,7), (5,1), (4,3), (3,5), (2,7), (1,9),\\
&(6,1), (5,3), (4,5), (3,7), (2,9), (1,11), (7,1), (6,3),\\
&(5,5), (4,7), (3,9), (2,11), (8,1), (7,3), (6,5), (5,7),\\
&(4,9), (3,11), (9,1), (8,3), (7,5), (6,7), (5,9), (9,3),\\
&(8,5), (7,7).
\end{aligned}
\]

Direct computation in Sage gives the decomposition
\[
\begin{aligned}
V(5,5)\otimes V(2,2)
={}&V(3,3)\oplus V(2,5)\oplus V(1,7)\oplus V(5,1)
\oplus 2V(4,3)\\
&\oplus 3V(3,5)\oplus 2V(2,7)\oplus V(1,9)
\oplus 2V(6,1)\\
&\oplus 4V(5,3)\oplus 4V(4,5)\oplus 4V(3,7)
\oplus 2V(2,9)\\
&\oplus V(1,11)\oplus 3V(7,1)\oplus 4V(6,3)
\oplus 5V(5,5)\\
&\oplus 4V(4,7)\oplus 3V(3,9)\oplus V(2,11)
\oplus 2V(8,1)\\
&\oplus 4V(7,3)\oplus 4V(6,5)\oplus 4V(5,7)
\oplus 2V(4,9)\\
&\oplus V(3,11)\oplus V(9,1)\oplus 2V(8,3)
\oplus 3V(7,5)\\
&\oplus 2V(6,7)\oplus V(5,9)\oplus V(9,3)
\oplus V(8,5)\oplus V(7,7).
\end{aligned}
\]

In particular, every dominant weight of the form
\[
5\rho+\beta,
\]
with $\beta$ a weight of $V(2\rho)$, appears as a component of
\[
V(5\rho)\otimes V(2\rho).
\]
\end{example}
\begin{example}
Let $\mathfrak g$ be of type $G_2$, and consider the tensor product
\[
V(5\rho)\otimes V(2\rho).
\]
Since $\rho=\omega_1+\omega_2$, this corresponds to
\[
V(5,5)\otimes V(2,2)
\]
in fundamental weight coordinates.

Using SageMath, we compute all dominant weights of the form
\[
\lambda = 5\rho+\beta,
\]
where $\beta$ is a weight of $V(2\rho)$. The resulting dominant weights are
\[
\begin{aligned}
&(3,3),(5,2),(7,1),(0,5),(2,4),(4,3),(6,2),(8,1),(10,0),\\
&(1,5),(3,4),(5,3),(7,2),(9,1),(11,0),(0,6),(2,5),(4,4),\\
&(6,3),(8,2),(10,1),(12,0),(1,6),(3,5),(5,4),(7,3),(9,2),\\
&(11,1),(13,0),(0,7),(2,6),(4,5),(6,4),(8,3),(10,2),(12,1),\\
&(14,0),(1,7),(3,6),(5,5),(7,4),(9,3),(11,2),(13,1),(15,0),\\
&(0,8),(2,7),(4,6),(6,5),(8,4),(10,3),(12,2),(14,1),(1,8),\\
&(3,7),(5,6),(7,5),(9,4),(11,3),(13,2),(15,1),(0,9),(2,8),\\
&(4,7),(6,6),(8,5),(10,4),(12,3),(14,2),(1,9),(3,8),(5,7),\\
&(7,6),(9,5),(11,4),(13,3),(0,10),(2,9),(4,8),(6,7),(8,6),\\
&(10,5),(3,9),(5,8),(7,7).
\end{aligned}
\]

Direct computation in Sage gives the decomposition
\[
\begin{aligned}
V(5,5)\otimes V(2,2)
={}&V(3,3)+V(5,2)+V(7,1)+V(0,5)+2V(2,4)+3V(4,3)\\
&+3V(6,2)+2V(8,1)+V(10,0)+3V(1,5)+6V(3,4)\\
&+7V(5,3)+6V(7,2)+4V(9,1)+2V(11,0)+3V(0,6)\\
&+8V(2,5)+11V(4,4)+11V(6,3)+9V(8,2)+6V(10,1)\\
&+3V(12,0)+8V(1,6)+14V(3,5)+16V(5,4)+15V(7,3)\\
&+11V(9,2)+7V(11,1)+3V(13,0)+5V(0,7)+14V(2,6)\\
&+19V(4,5)+19V(6,4)+16V(8,3)+11V(10,2)+6V(12,1)\\
&+2V(14,0)+11V(1,7)+18V(3,6)+21V(5,5)+19V(7,4)\\
&+15V(9,3)+9V(11,2)+4V(13,1)+V(15,0)+5V(0,8)\\
&+14V(2,7)+19V(4,6)+19V(6,5)+16V(8,4)+11V(10,3)\\
&+6V(12,2)+2V(14,1)+8V(1,8)+14V(3,7)+16V(5,6)\\
&+15V(7,5)+11V(9,4)+7V(11,3)+3V(13,2)+V(15,1)\\
&+3V(0,9)+8V(2,8)+11V(4,7)+11V(6,6)+9V(8,5)\\
&+6V(10,4)+3V(12,3)+V(14,2)+3V(1,9)+6V(3,8)\\
&+7V(5,7)+6V(7,6)+4V(9,5)+2V(11,4)+V(13,3)\\
&+V(0,10)+2V(2,9)+3V(4,8)+3V(6,7)+2V(8,6)\\
&+V(10,5)+V(3,9)+V(5,8)+V(7,7).
\end{aligned}
\]

In particular, every dominant weight of the form
\[
5\rho+\beta,
\]
with $\beta$ a weight of $V(2\rho)$, appears as a component of
\[
V(5\rho)\otimes V(2\rho).
\]
\end{example}

\begin{example}
Let $\mathfrak g$ be the simple Lie algebra of type $G_2$, and let
\[
\rho=\omega_1+\omega_2.
\]
We consider the tensor product
\[
V(4\rho)\otimes V(3\rho)
=
V(4,4)\otimes V(3,3),
\]
where weights are expressed in the basis of fundamental weights.

According to Conjecture~\ref{RekhaConj}, the irreducible module
$V(a,b)$ occurs in
\[
V(4\rho)\otimes V(3\rho)
\]
if and only if
\[
(a,b)=(4,4)+\beta
\]
for some weight $\beta\in \mathrm{wt}(V(3,3))$.

The dominant weights predicted by the conjecture are therefore:
\[
\begin{aligned}
&(1,1),
(3,0),
(0,2),
(2,1),
(4,0),
(1,2),
(3,1),
(5,0),
(0,3),
(2,2),\\
&(4,1),
(6,0),
(1,3),
(3,2),
(5,1),
(7,0),
(0,4),
(2,3),
(4,2),
(6,1),\\
&(8,0),
(1,4),
(3,3),
(5,2),
(7,1),
(9,0),
(0,5),
(2,4),
(4,3),
(6,2),\\
&(8,1),
(10,0),
(1,5),
(3,4),
(5,3),
(7,2),
(9,1),
(11,0),
(0,6),
(2,5),\\
&(4,4),
(6,3),
(8,2),
(10,1),
(12,0),
(1,6),
(3,5),
(5,4),
(7,3),
(9,2),\\
&(11,1),
(13,0),
(0,7),
(2,6),
(4,5),
(6,4),
(8,3),
(10,2),
(12,1),
(14,0),\\
&(1,7),
(3,6),
(5,5),
(7,4),
(9,3),
(11,2),
(13,1),
(15,0),
(0,8),
(2,7),\\
&(4,6),
(6,5),
(8,4),
(10,3),
(12,2),
(14,1),
(16,0),
(1,8),
(3,7),
(5,6),\\
&(7,5),
(9,4),
(11,3),
(13,2),
(15,1),
(17,0),
(0,9),
(2,8),
(4,7),
(6,6),\\
&(8,5),
(10,4),
(12,3),
(14,2),
(16,1),
(1,9),
(3,8),
(5,7),
(7,6),
(9,5),\\
&(11,4),
(13,3),
(0,10),
(2,9),
(4,8),
(6,7),
(8,6),
(10,5),
(1,10),
(3,9),\\
&(5,8),
(7,7).
\end{aligned}
\]

Using {\tt SageMath}, we compute:
\[
\begin{aligned}
V(3,3)\otimes V(4,4)
={}&\,V(1,1)+V(3,0)+V(0,2)+3V(2,1)+3V(4,0)+5V(1,2)+7V(3,1)\\
&+6V(5,0)+3V(0,3)+10V(2,2)+13V(4,1)+9V(6,0)+11V(1,3)\\
&+17V(3,2)+19V(5,1)+13V(7,0)+6V(0,4)+19V(2,3)+27V(4,2)\\
&+25V(6,1)+16V(8,0)+18V(1,4)+29V(3,3)+35V(5,2)+32V(7,1)\\
&+18V(9,0)+9V(0,5)+28V(2,4)+41V(4,3)+42V(6,2)+35V(8,1)\\
&+20V(10,0)+23V(1,5)+38V(3,4)+48V(5,3)+48V(7,2)+36V(9,1)\\
&+19V(11,0)+10V(0,6)+32V(2,5)+48V(4,4)+51V(6,3)+47V(8,2)\\
&+35V(10,1)+17V(12,0)+23V(1,6)+38V(3,5)+49V(5,4)+51V(7,3)\\
&+42V(9,2)+29V(11,1)+14V(13,0)+9V(0,7)+28V(2,6)+42V(4,5)\\
&+45V(6,4)+43V(8,3)+35V(10,2)+22V(12,1)+10V(14,0)+18V(1,7)\\
&+29V(3,6)+37V(5,5)+39V(7,4)+32V(9,3)+24V(11,2)+15V(13,1)\\
&+6V(15,0)+6V(0,8)+19V(2,7)+28V(4,6)+29V(6,5)+27V(8,4)\\
&+22V(10,3)+14V(12,2)+8V(14,1)+3V(16,0)+11V(1,8)+17V(3,7)\\
&+21V(5,6)+21V(7,5)+16V(9,4)+11V(11,3)+7V(13,2)+3V(15,1)\\
&+V(17,0)+3V(0,9)+10V(2,8)+14V(4,7)+13V(6,6)+11V(8,5)\\
&+8V(10,4)+4V(12,3)+2V(14,2)+V(16,1)+5V(1,9)+7V(3,8)\\
&+8V(5,7)+7V(7,6)+4V(9,5)+2V(11,4)+V(13,3)+V(0,10)\\
&+3V(2,9)+4V(4,8)+3V(6,7)+2V(8,6)+V(10,5)+V(1,10)\\
&+V(3,9)+V(5,8)+V(7,7).
\end{aligned}
\]
Hence every dominant weight predicted by the conjecture occurs in the tensor product decomposition, and no additional dominant weights occur. This provides strong computational evidence for Conjecture~\ref{RekhaConj} in type $G_2$.
\end{example}

\begin{remark}
The preceding examples also provide evidence for the Schur positivity phenomenon predicted by Proposition~2. Explicit computations in type $G_2$ show that every irreducible summand occurring in
\[
V(5,5)\otimes V(2,2)
\]
also occurs in
\[
V(3,3)\otimes V(4,4),
\]
with multiplicity at least as large. Consequently,
\[
\mathrm{ch}\,V(3,3)\,\mathrm{ch}\,V(4,4)
-
\mathrm{ch}\,V(5,5)\,\mathrm{ch}\,V(2,2)
\in \mathbb{Z}_{\geq 0}[P]^W.
\]

More precisely, for every dominant weight $\lambda$,
\[
[V(\lambda):V(3,3)\otimes V(4,4)]
\geq
[V(\lambda):V(5,5)\otimes V(2,2)].
\]
Thus the Schur positivity inclusion predicted by Proposition~2 holds in these examples.
\end{remark}

\end{document}